\documentclass[11pt,a4paper]{amsart}
\usepackage{amsmath,amssymb,amsthm,amsfonts}
\usepackage{mathrsfs}
\usepackage{amscd}
\usepackage{listings}
\usepackage{paralist}
\usepackage{booktabs}
\usepackage[usenames,dvipsnames]{color}
\usepackage{comment}
\usepackage{graphicx}
\usepackage{latexsym}
\usepackage[shortlabels]{enumitem}
\usepackage{thmtools}
\usepackage{url}
\usepackage[all,cmtip]{xy}
\usepackage[bookmarks=false]{hyperref} %hyperref wants to be loaded last

\newtheorem{theorem}{Theorem}[section]

\newtheorem{remark}[theorem]{Remark}
\newtheorem{cor}[theorem]{Corollary}
\newenvironment{corollary}{\begin{cor} \em}{\end{cor}}
\newtheorem{tont}[theorem]{Definition}
\newenvironment{definition}{\begin{tont} \em}{\end{tont}}
\newtheorem{lemma}[theorem]{Lemma}

\newtheorem{proposition}[theorem]{Proposition}

\begin{document}

\title [A valuation theorem for Noetherian rings]{A valuation theorem for Noetherian rings}
\author{Antoni Rangachev}
\begin{abstract} 
Let $\mathcal{A} \subset \mathcal{B}$ be integral domains. Suppose $\mathcal{A}$ is Noetherian and $\mathcal{B}$ is a finitely generated $\mathcal{A}$-algebra. Denote by $\overline{\mathcal{A}}$ the integral closure of $\mathcal{A}$ in $\mathcal{B}$. We show that $\overline{\mathcal{A}}$ is determined by finitely many unique  discrete valuation rings. Our result generalizes Rees' classical valuation theorem for ideals. We also obtain a variant of Zariski's main theorem. 
\end{abstract}
\subjclass[2010]{13A18, 13B22, 13A30, 14A15, (14B05).}
\keywords{Integral closure of rings, discrete valuation rings, Rees valuations, Chevalley's constructability result, Zariski's main theorem}
\address{Department of Mathematics\\
 University of Chicago \\
 Chicago, IL 60637\\
Institute of Mathematics \\
and Informatics, Bulgarian Academy of Sciences\\
Akad. G. Bonchev 8, Sofia 1113, Bulgaria}
\maketitle
%\selectlanguage{english}
\section{Introduction} 
Let $\mathcal{A} \subset \mathcal{B}$ be integral domains. Denote the integral closure of $\mathcal{A}$ in $\mathcal{B}$ by $\overline{\mathcal{A}}$. 
Suppose there exist  valuation rings $\mathcal{V}_1, \ldots, \mathcal{V}_r$ in $\mathrm{Frac}(\overline{\mathcal{A}})$ such that 
\begin{equation}\label{decompos}
\overline{\mathcal{A}}= \cap_{i=1}^r \mathcal{V}_i \cap \mathcal{B},
\end{equation}
where the intersection takes place in $\mathrm{Frac}(\mathcal{B})$. We say that (\ref{decompos}) is a {\it valuation decomposition} of $\overline{\mathcal{A}}$. We say the decomposition is {\it irredundant} or {\it minimal} if dropping any $\mathcal{V}_i$ violates (\ref{decompos}). The main result of this paper is the following valuation theorem. 
%Let $\mathcal{B}'$ be the maximal subring of $\mathcal{B}$ with $\mathrm{Frac}(\mathcal{B}')=\mathrm{Frac}(\mathcal{A})$

\begin{theorem}\label{main} Suppose $\mathcal{A}$ is Noetherian and $\mathcal{B}$ is a finitely generated $\mathcal{A}$-algebra. Then one of the following holds:
\begin{enumerate}
    \item [\rm{(i)}] $\overline{\mathcal{A}}=\mathcal{B}$;
    \item [\rm{(ii)}]  $\mathrm{Ass}_{\overline{\mathcal{A}}}(\mathcal{B}/\overline{\mathcal{A}})= \{(0)\}$;
    \item [\rm{(iii)}]  There exist unique discrete valuation rings $\mathcal{V}_1, \ldots, \mathcal{V}_r$ in $\mathrm{Frac}(\overline{\mathcal{A}})$ such that $\overline{\mathcal{A}}= \cap_{i=1}^r \mathcal{V}_i \cap \mathcal{B}$ is minimal. Furthermore, if $\mathcal{A}$ is locally formally equidimensional, then each $\mathcal{V}_i$ is a divisorial valuation ring with respect to a Noetherian subring of $\overline{\mathcal{A}}$.
\end{enumerate}
\end{theorem}
It's well-known that $\overline{\mathcal{A}}$ may fail to be Noetherian \cite[Ex.\ 4.10]{Huneke}. The proof of Thm.\ \ref{main} rests upon three key observations. First, we show that  there exists $f \in \overline{\mathcal{A}}$ such that $\overline{\mathcal{A}}_f$ is Noetherian. Then we use this to prove that $\mathrm{Ass}_{\overline{\mathcal{A}}}(\mathcal{B}/\overline{\mathcal{A}})$ is finite by results of \cite{Rangachev2}. We set each $\mathcal{V}_i$ to be the localization of $\overline{\mathcal{A}}$ at a prime  in $\mathrm{Ass}_{\overline{\mathcal{A}}}(\mathcal{B}/\overline{\mathcal{A}})$. Then $\mathcal{V}_i$ is a DVR by \cite[Thm.\ 1.1 (i)]{Rangachev2}. Finally, to get 
the equality in (\ref{decompos}) we show that the minimal primes of an ideal in $\overline{\mathcal{A}}$ which is the annihilator of an element of $\mathcal{B}/\overline{\mathcal{A}}$ are in $\mathrm{Ass}_{\overline{\mathcal{A}}}(\mathcal{B}/\overline{\mathcal{A}})$. 
As another application of these observations we obtain a variant of Zariski's main theorem. 

Let $R$ be a Noetherian domain. Suppose $\mathcal{A}=\oplus_{i=0}^{\infty}\mathcal{A}_i \subset \mathcal{B}=\oplus_{i=0}^{\infty}\mathcal{B}_i$ is a homogeneous inclusion of graded Noetherian domains with $\mathcal{A}_0=\mathcal{B}_0=R$. Suppose $\mathcal{B}$ is a finitely generated $\mathcal{A}$-algebra. For each $n$ denote by $\overline{\mathcal{A}_n}$ the integral closure of $\mathcal{A}_n$ in $\mathcal{B}_n$. It's the $R$-module consisting of all elements in $\mathcal{B}_n$ that are integral over $\mathcal{A}$. For the discrete valuations $\mathcal{V}_i$ in Thm.\ \ref{main} set $V_i: = \mathcal{V}_i \cap \mathrm{Frac}(R)$. Define $\mathcal{A}_{n}V_i \cap \mathcal{B}_n$ to be the set of elements in $\mathcal{B}_n$ that map to $\mathcal{A}_{n}V_i$ as a submodule of $\mathcal{B}_{n}V_i$. The following is a corollary to our main result.

\begin{corollary}\label{graded}
 Suppose $\mathcal{A}V=\overline{\mathcal{A}}V$ for each valuation $V$ in $\mathrm{Frac}(R)$. Then one of the following holds:
\begin{enumerate}
    \item [\rm{(i)}] $\overline{\mathcal{A}}=\mathcal{B}$;
    \item [\rm{(ii)}]  $\mathrm{Ass}_{\overline{\mathcal{A}}}(\mathcal{B}/\overline{\mathcal{A}})= \{(0)\}$;
    \item [\rm{(iii)}]
For each $n$ we have the following valuation decomposition $\overline{\mathcal{A}_n}= \cap_{i=1}^{r} \mathcal{A}_{n}V_{i} \cap \mathcal{B}_n$. Furthermore, if $V_i \neq \mathrm{Frac}(R)$ for $i=1, \ldots, r$, then the valuation decomposition is minimal and the $V_i$s are unique. 
\end{enumerate}
\end{corollary}
Let $I$ be an ideal in a Noetherian domain $R$. Let $t$ be a variable. The graded algebra $R[It]:= R \oplus It \oplus I^2t^2 \oplus \cdots$ is called the {\it Rees algebra} of $I$. It's contained in the polynomial ring $R[t]:= R \oplus Rt \oplus Rt^2 \oplus \cdots$. For each $n$ denote by $\overline{I^n}$ the integral closure of $I^n$ in $R$. Set $\mathcal{A}:=R[It]$ and $\mathcal{B}:=R[t]$ in Cor.\ \ref{graded}. Note that for each valuation ring $V$ in $\mathrm{Frac}(R)$ we have $\mathcal{A}V=V[t]$ or $\mathcal{A}V=V[at]$ where $IV=(a)$ for some $a \in I$. Thus $\mathcal{A}V$ is integrally closed, and so $\mathcal{A}V=\overline{\mathcal{A}}V$.  Cor.\ \ref{graded} recovers a classical result due to Rees \cite{Rees56}. 
\begin{corollary}[Rees' valuation theorem]\label{Rees} Let $R$ be a Noetherian domain and $I$ be a nonzero ideal in $R$. There exists unique discrete valuations $V_1, \ldots, V_r$ in the field of fractions of $R$ such that $\overline{I^n}= \cap_{i=1}^r I^{n}V_i \cap R$  for each $n$. \end{corollary}

In the setting of Cor.\ \ref{graded} assume additionally that $R$ is locally formally equidimensional.
We can give a geometric interpretation of the centers of the $V_i$s in $R$ using Chevalley's constructability result as follows.  Consider the structure map $c \colon \mathrm{Proj}(\mathcal{A}) \rightarrow \mathrm{Spec}(R)$. For each integer $l \geq 0$ set $$S(l):= \{\mathfrak{p} \in \mathrm{Spec}(R) \colon \dim \mathrm{Proj}(\mathcal{A} \otimes_{R} k(\mathfrak{p})) \geq l\}.$$
By Chevalley's \cite[Thm.\ 13.1.3 and Cor.\ 13.1.5]{Grothendieck1} $S(l)$ is closed in $\mathrm{Spec}(R)$. For $i=1, \ldots, r$ denote by $\mathfrak{m}_{i}$ the center of $V_i$ in $R$. Set $e:= \dim \mathrm{Proj}(\mathcal{A} \otimes_{R} \mathrm{Frac}(R))$. 
\begin{theorem}\label{Chevalley} Suppose $R$ is locally formally equidimensional. If $\mathrm{ht}(\mathfrak{m}_{i})>1$ for some $i$, then $\mathfrak{m}_{i}$ is a minimal prime of $S(\mathrm{ht}(\mathfrak{m}_{i})+e-1)$. 
\end{theorem}

{\bf  Acknowledgements.} I thank Madhav Nori and Bernard Teissier for stimulating and helpful conversations. I was partially supported by the University of Chicago FACCTS grant ``Conormal and Arc Spaces in the Deformation Theory of Singularities.''

\section{Proofs}

The proof of Thm.\ \ref{main} is based on three key propositions. 

\begin{proposition}\label{Noether} Suppose $\mathcal{A} \subset \mathcal{B}$ are integral domains. Suppose $\mathcal{A}$ is Noetherian and $\mathcal{B}$ is a finitely generated $\mathcal{A}$-algebra. Then there exists $f \in \overline{\mathcal{A}}$ such that $\overline{\mathcal{A}}_f$ is Noetherian.
\end{proposition}
\begin{proof} Denote by $E$ the algebraic closure of $\mathrm{Frac}(\mathcal{A})$ in $\mathrm{Frac}(\mathcal{B})$. By Zariski's lemma $E$ is a finite field extension of $\mathrm{Frac}(\mathcal{A})$. Because $E=\mathrm{Frac}(\overline{\mathcal{A}})$, there exist $f_1, \ldots, f_k \in \overline{\mathcal{A}}$ such that $E=\mathrm{Frac}(\mathcal{A})(f_1, \ldots, f_k)$. Set $\mathcal{A}':=\mathcal{A}[f_1, \ldots, f_k]$. Then $\mathcal{A}'$ is Noetherian and  $\mathrm{Frac}(\mathcal{A}')=\mathrm{Frac}(\overline{\mathcal{A}})$. By \cite[Prp.\ 2.1]{Rangachev2} $\mathrm{Ass}_{\mathcal{A}'}(\mathcal{B}/\mathcal{A}')$ is finite. But $\mathrm{Ass}_{\mathcal{A}'}(\overline{\mathcal{A}}/\mathcal{A}') \subset \mathrm{Ass}_{\mathcal{A}'}(\mathcal{B}/\mathcal{A}')$. So $\mathrm{Ass}_{\mathcal{A}'}(\overline{\mathcal{A}}/\mathcal{A}')$ is finite, too. Select $f \in \mathcal{A}'$ from the intersection of all minimal primes in $\mathrm{Ass}_{\mathcal{A}'}(\overline{\mathcal{A}}/\mathcal{A}')$. Then $\mathcal{A}'_f=\overline{\mathcal{A}}_f$; hence $\overline{\mathcal{A}}_f$ is Noetherian.
\end{proof}

The next proposition strengthens  \cite[Thm.\ 1.1 \rm{(ii)}]{Rangachev2} in the domain case. 
\begin{proposition}\label{finite} Suppose $\mathcal{A} \subset \mathcal{B}$ are integral domains. Suppose $\mathcal{A}$ is Noetherian and $\mathcal{B}$ is a finitely generated $\mathcal{A}$-algebra. Then $\mathrm{Ass}_{\overline{\mathcal{A}}}(\mathcal{B}/\overline{\mathcal{A}})$ and  $\mathrm{Ass}_{\mathcal{A}}(\mathcal{B}/\overline{\mathcal{A}})$ are finite. 
\end{proposition}
\begin{proof} By Prp.\ \ref{Noether} there exists $f \in \overline{\mathcal{A}}$ such that $\overline{\mathcal{A}}_f$ is Noetherian. Let $\mathfrak{q} \in \mathrm{Ass}_{\overline{\mathcal{A}}}(\mathcal{B}/\overline{\mathcal{A}})$. If $f \not \in \mathfrak{q}$, then $\mathfrak{q} \in \mathrm{Ass}_{\overline{\mathcal{A}}_f}(\mathcal{B}_f/\overline{\mathcal{A}}_f)$. The last set is finite by \cite[Prp.\ 2.1]{Rangachev2}. Suppose $f \in \mathfrak{q}$. As before, denote by $E$ the algebraic closure of $\mathrm{Frac}(\mathcal{A})$ in $\mathrm{Frac}(\mathcal{B})$. It's a finite field extension of $\mathrm{Frac}(\mathcal{A})$. Denote by $L$ the integral closure of $\mathcal{A}$ in $E$. By the Mori--Nagata Theorem $L$ is a Krull domain (\cite[Prp.\ 12, pg.\ 209]{Bourbaki} and \cite[Ex.\ 4.15]{Huneke}). But $L$ is also the integral closure of $\overline{\mathcal{A}}$ in its field of fractions. Let $\mathfrak{q}'$ be a prime in $L$ that contracts to $\mathfrak{q}$. We have $\overline{\mathcal{A}}_{\mathfrak{q}} \subset L_{\mathfrak{q}'}$. By \cite[Thm.\ 1.1 \rm{(i)}]{Rangachev2} $\overline{\mathcal{A}}_{\mathfrak{q}}$ is a DVR. As $\overline{\mathcal{A}}$ and $L$ have the same field of fractions, $\overline{\mathcal{A}}_{\mathfrak{q}} = L_{\mathfrak{q}'}$. Thus $\mathrm{ht}(\mathfrak{q}')=1$. Because $L$ is a Krull domain, there are finitely many height one prime ideals in $L$ containing $f$. Thus there are finitely many $\mathfrak{q} \in \mathrm{Ass}_{\overline{\mathcal{A}}}(\mathcal{B}/\overline{\mathcal{A}})$ containing $f$. This proves the finiteness of $\mathrm{Ass}_{\overline{\mathcal{A}}}(\mathcal{B}/\overline{\mathcal{A}})$. Alternatively, apply directly \cite[Thm.\ 1.1 \rm{(ii)}]{Rangachev2} for $\mathcal{A}'$ and $\mathcal{B}$ noting that $\mathcal{A}$ and $\mathcal{A}'$ have the same integral closure in $\mathcal{B}$.

Let $\mathfrak{p} \in \mathrm{Ass}_{\mathcal{A}}(\mathcal{B}/\overline{\mathcal{A}})$. If $f \not \in \mathfrak{p}$, then $\mathfrak{p}$ is a contraction from a prime in  $\mathrm{Ass}_{\overline{\mathcal{A}}_f}(\mathcal{B}_f/\overline{\mathcal{A}}_f)$ which is finite by \cite[Prp.\ 2.1]{Rangachev2}.
If $f \in \mathfrak{p}$, then the proof of \cite[Thm.\ 1.1 \rm{(ii)}]{Rangachev2} shows that  $\mathfrak{p} \in \mathrm{Ass}_{\mathcal{A}}(\mathcal{A}/f\mathcal{A})$ which is finite because $\mathcal{A}$ is Noetherian. The proof is now complete. 
\end{proof}
%\begin{remark} \rm Here we give a more detailed explanation why the hypothesis listed in the Introduction guarantee that $\overline{\mathcal{A}}$ is generically Noetherian. Suppose $\mathcal{A}_{(0)}=\overline{\mathcal{A}}_{(0)}$. In the proof of \cite[Thm.\ 1.1 \rm{(ii)}]{Rangachev2} we show the existence of $f \in \mathcal{A}$ such that $\mathcal{A}_f=\overline{\mathcal{A}}_f$. Then $\overline{\mathcal{A}}_f$ is Noetherian because $\mathcal{A}$ is Noetherian. Suppose there exits $f \in \mathcal{A}$ such that $\mathcal{A}_f$ is normal and $\mathrm{char}(\mathrm{Frac}(\mathcal{A}))=0$. Then by (\cite[\href{http://stacks.math.columbia.edu/tag/032M}{Tag 032M}]{Stacks} $L_f$ is module-finite over $\mathcal{A}_f$. Thus $\overline{\mathcal{A}}_f$ is module-finite over $\mathcal{A}_f$; hence it's Noetherian. Next, suppose there exists prime $\mathfrak{q}$ in $\mathcal{B}$, $\mathfrak{p}:=\mathcal{A} \cap \mathfrak{q}$, such that $\mathcal{B}_{\mathfrak{q}}/\mathfrak{p}\mathcal{B}_{\mathfrak{q}}$ is module-finite over the residue field $\kappa(\mathfrak{p})$. Then by Zariski's main theorem (\cite[\href{http://stacks.math.columbia.edu/tag/00Q9}{Tag 00Q9}]{Stacks} or \cite[Ex.\ 4.26]{Huneke}) there exists $f \in \overline{\mathcal{A}}$ such that $\overline{\mathcal{A}}_f=\mathcal{B}_f$; in particular, $\overline{\mathcal{A}}_f$ is Noetherian. 
%\end{remark}

\begin{proposition}\label{minimal}
Suppose $\mathcal{A} \subset \mathcal{B}$ are integral domains. Suppose $\mathcal{A}$ is Noetherian and $\mathcal{B}$ is a finitely generated $\mathcal{A}$-algebra. Let $b \in \mathcal{B}$ be such that $J:= (\overline{\mathcal{A}} : _{\overline{\mathcal{A}}}  b)$ is a nonunit ideal in $\overline{\mathcal{A}}$. Then the minimal primes of $J$ are in $\mathrm{Ass}_{\overline{\mathcal{A}}}(\mathcal{B}/\overline{\mathcal{A}})$.
\end{proposition}
\begin{proof}
If $J=(0)$, then clearly $J \in \mathrm{Ass}_{\overline{\mathcal{A}}}(\mathcal{B}/\overline{\mathcal{A}})$. Suppose $J \neq (0)$. Select a nonzero $h \in J$. Then $J:= ((h) : _{\overline{\mathcal{A}}}  hb)$. Thus the minimal primes of $J$ are among the minimal primes of $(h)$ each of which is of height one. Denote by $L$ the integral closure of $\overline{\mathcal{A}}$ in $\mathrm{Frac}(\overline{\mathcal{A}})$. Because $L$ is a Krull domain, then there are finitely many minimal primes of $hL$. But $L$ is integral over $\overline{\mathcal{A}}$. So by incomparability each minimal prime of $(h)$
is a contraction of a prime of height one in $L$ which has to be a minimal prime of $hL$. Therefore, $(h)$ has finitely many minimal primes, and so does $J$.

Denote by $\mathfrak{q}_1, \ldots, \mathfrak{q}_l$ the minimal primes of $J$. First, we want to show that for each $1 \leq i \leq l$ there exists a positive integer $s_i$ such that $\mathfrak{q}_{i}^{s_i} \subset J\overline{\mathcal{A}}_{\mathfrak{q}_i}$. We proceed as in the proof of \cite[Thm.\ 1.1 \rm{(i)}]{Rangachev2}. Set $\mathfrak{p}_i:=\mathfrak{q}_i \cap \mathcal{A}$. We can assume that $\mathcal{A}$ is local at $\mathfrak{p}_i$. Let $\widehat{\mathcal{A}}$ be the completion of $\mathcal{A}$ with respect to $\mathfrak{p}_i$. Set $\mathcal{A}':=\overline{\mathcal{A}} \otimes_{\mathcal{A}} \widehat{\mathcal{A}}$ and $\mathcal{B}':= \mathcal{B} \otimes_{\mathcal{A}} \widehat{\mathcal{A}}$. 
Replace $\widehat{\mathcal{A}}, \mathcal{A}'$ and $\mathcal{B}'$ by their reduced structures.
Because $\widehat{\mathcal{A}}$ is a reduced complete local ring and $\mathcal{B}'$ is a finitely generated $\widehat{\mathcal{A}}$-algebra, then 
by \cite[\href{http://stacks.math.columbia.edu/tag/03GH}{Tag 03GH}]{Stacks} $\mathcal{A}'$ is module-finite over $\widehat{\mathcal{A}}$. In particular, $\mathcal{A}'$ is Noetherian. Clearly, $J\mathcal{A}'$ is primary to $\mathfrak{q}_{i}\mathcal{A}'$. Thus there exists $s_i$ such that $\mathfrak{q}_{i}^{s_i}\mathcal{A}' \subset J\mathcal{A}'$. Hence $\mathfrak{q}_{i}^{s_i}b \in \mathcal{A}'$. But $\mathfrak{q}_{i}^{s_i}b \in \mathcal{B}$. Thus by \cite[Prp.\ 2.2]{Rangachev2} $\mathfrak{q}_{i}^{s_i}b \in \overline{\mathcal{A}}_{\mathfrak{p}_i}$, and so  $\mathfrak{q}_{i}^{s_i}b \in \overline{\mathcal{A}}_{\mathfrak{q}_i}$. This implies $\mathfrak{q}_{i}^{s_i} \subset J\overline{\mathcal{A}}_{\mathfrak{q}_i}$ by \cite[Prp.\ 3.14]{Atiyah} applied for $(\overline{\mathcal{A}},b)/\overline{\mathcal{A}}$.

Assume that the $s_i$ defined above are the minimal possible. Fix $1 \leq j \leq l$. For each $i \neq j$ by prime avoidance we can select $c_i \in \mathfrak{q}_{i}^{s_i}$ and $c_i \not \in \mathfrak{q}_j$. Let $c_j \in \mathfrak{q}_{j}^{s_{j}-1}$ with $c_j \not \in J\overline{\mathcal{A}}_{\mathfrak{q}_j}$. Set $c:=c_1 \cdots c_l$. Then $\mathfrak{q}_j = (\overline{\mathcal{A}} : _{\overline{\mathcal{A}}}  cb)$ and thus $\mathfrak{q}_j \in \mathrm{Ass}_{\overline{\mathcal{A}}}(\mathcal{B}/\overline{\mathcal{A}})$.
\end{proof}
\vspace{.1cm}
\begin{center}
{\it Proof of Theorem \ref{main}}
\end{center}
\vspace{.2cm}
We can proceed with the proof of Thm.\ \ref{main}. Suppose  $\overline{\mathcal{A}} \neq \mathcal{B}$ and $\mathrm{Ass}_{\overline{\mathcal{A}}}(\mathcal{B}/\overline{\mathcal{A}}) \neq \{(0)\}$. Then by Prp.\ \ref{finite} $\mathrm{Ass}_{\overline{\mathcal{A}}}(\mathcal{B}/\overline{\mathcal{A}})$ contains finitely many nonzero prime ideals which we denote by $\mathfrak{q}_1, \ldots, \mathfrak{q}_r$. By \cite[Thm.\ 1.1 \rm{(i)}]{Rangachev2} $\mathcal{V}_i:=\overline{\mathcal{A}}_{\mathfrak{q}_i}$ is a DVR for each $i=1, \ldots, r$.  Obviously, $\overline{\mathcal{A}} \subseteq \cap_{i=1}^{r} \mathcal{V}_i \cap \mathcal{B}.$ Let $b=x/y \in  \cap_{i=1}^{r} \mathcal{V}_i \cap \mathcal{B}.$ Set $J:= (y\overline{\mathcal{A}} : _{\overline{\mathcal{A}}} x)$. We have $J\mathcal{V}_i=\mathcal{V}_i$ for each $i$. Thus $J \nsubseteq \mathfrak{q}_i$ for each $i$. But $Jb \in \overline{\mathcal{A}}$ and $J \neq (0)$. So by Prp.\ \ref{minimal} if $J$ is a nonunit ideal, its minimal primes are among $\mathfrak{q}_1, \ldots, \mathfrak{q}_r$ which is impossible. Thus $J$ has to be the unit ideal, which implies that $b \in \overline{\mathcal{A}}$. 

To prove minimality of the valuation decomposition, suppose we can drop $\mathcal{V}_j$ in (\ref{decompos}) for some $j$. Then localizing both sides of (\ref{decompos}) at $\mathfrak{q}_j$ we obtain that $\overline{\mathcal{A}}_{\mathfrak{q}_j}=\mathcal{B}_{\mathfrak{q}_j}$ which contradicts with $\mathfrak{q}_j \in \mathrm{Supp}(\mathcal{B}/\overline{\mathcal{A}})$. Thus (\ref{decompos}) is minimal. 
We are left with proving the uniqueness of the $\mathcal{V}_i$s. Suppose 
\begin{equation}\label{2decompos}
\overline{\mathcal{A}}= \cap_{j=1}^s \mathcal{V}_{j}' \cap \mathcal{B}
\end{equation}
is a minimal discrete valuation decomposition. Set $\mathcal{B}':=\mathrm{Frac}(\overline{\mathcal{A}}) \cap \mathcal{B}$. We have $\overline{\mathcal{A}}_{\mathfrak{q}_i} \neq \mathcal{B}_{\mathfrak{q}_i}'$ for each $i=1, \ldots, r$. As the intersection in (\ref{2decompos}) takes place in $\mathcal{B}'$ we can replace in it $\mathcal{B}$ by $\mathcal{B}'$. Localizing both sides of (\ref{2decompos}) at $\mathfrak{q}_1$ we obtain that there exists a valuation $\mathcal{V}_{l}'$ such that $(\mathcal{V}_{l}')_{\mathfrak{q}_1} \neq \mathrm{Frac}(\overline{\mathcal{A}})$. But $\mathcal{V}_1=\overline{\mathcal{A}}_{\mathfrak{q}_1} \subset (\mathcal{V}_{l}')_{\mathfrak{q}_1}$. Thus $\mathcal{V}_1=(\mathcal{V}_{l}')_{\mathfrak{q}_1}$. Also, $\mathcal{V}_{l}' \subset (\mathcal{V}_{l}')_{\mathfrak{q}_1}$, and so $\mathcal{V}_{l}' = (\mathcal{V}_{l}')_{\mathfrak{q}_1}$. Therefore, $\mathcal{V}_1=\mathcal{V}_{l}'$. Continuing this process we obtain that each $\mathcal{V}_i$ appears in (\ref{2decompos}). As (\ref{2decompos}) is minimal, we obtain that $s=r$ and after possibly renumbering we get $\mathcal{V}_{i}=\mathcal{V}_{i}'$ for $i=1, \ldots, r$. 

Let $\mathcal{A}'$ be the module-finite $\mathcal{A}$-algebra defined in the proof of of Prp.\ \ref{Noether}. Suppose $\mathcal{A}$ is locally formally equidimensional. Then so is $\mathcal{A}'$. Note that $\mathrm{Frac}(\mathcal{A}')=\mathrm{Frac}(\overline{\mathcal{A}})$. Denote by $\mathfrak{m}_{\mathcal{V}_i}$ the maximal ideal of $\mathcal{V}_i$. Set $\mathfrak{p}_{i}:=\mathfrak{m}_{\mathcal{V}_i}\cap \mathcal{A}'$. By Cohen's dimension inequality (see \cite[Thm.\ B.2.5]{Huneke})
$$\mathrm{tr.\ deg}_{\kappa(\mathfrak{p}_i)}\kappa(\mathfrak{m}_{\mathcal{V}_i}) \leq \mathrm{ht}(\mathfrak{p}_i)-1.$$
Because $\mathfrak{p}_i=\mathfrak{q}_i \cap \mathcal{A}'$, then by \cite[Thm.\ 1.1 \rm{(iii)}]{Rangachev2} we get $\mathrm{ht}(\mathfrak{p}_i)=1$. Therefore, $\mathrm{tr.\ deg}_{\kappa(\mathfrak{p}_i)}\kappa(\mathfrak{m}_{\mathcal{V}_i}) = \mathrm{ht}(\mathfrak{p}_i)-1=0$. Hence each $\mathcal{V}_i$ is a divisorial valuation ring in $\mathrm{Frac}(\overline{\mathcal{A}})$.
\qedsymbol
\begin{remark}
\rm{As it's well-known, an integrally closed domain equals the intersection of all valuation rings in its field of fractions that contain it. From here one derives set-theoretically that $\overline{\mathcal{A}}=\cap \mathcal{V} \cap \mathcal{B}$ where the intersection is taken over all valuation rings in $\mathrm{Frac}(\overline{\mathcal{A}})$ that contain the integral closure of $\overline{\mathcal{A}}$ in $\mathrm{Frac}(\overline{\mathcal{A}})$. 
Because $\mathcal{A}'$ and $\overline{\mathcal{A}}$ have the same integral closure and $\mathcal{A}'$ is Noetherian (see the proof of Prp.\ \ref{Noether}), then in the intersection we can take only DVRs. Thus, the real contribution of Thm.\ \ref{main} is that under the additional hypothesis that $\mathcal{B}$ is a finitely generated $\mathcal{A}$-algebra, one can take finitely many uniquely determined DVRs each of which is a localization of $\overline{\mathcal{A}}$ at a height one prime ideal.}
\end{remark}
\vspace{.1cm}
\begin{center}
{\it Proof of Cor.\ \ref{graded} and Cor.\ \ref{Rees}}
\end{center}
\vspace{.2cm}
Suppose  $\mathrm{Ass}_{\overline{\mathcal{A}}}(\mathcal{B}/\overline{\mathcal{A}}) \neq \{(0)\}$ and  $\overline{\mathcal{A}} \neq \mathcal{B}$. Because $\mathcal{A}V=\overline{\mathcal{A}}V$ and $\overline{\mathcal{A}} \subset \overline{\mathcal{A}}V$ we get $\overline{\mathcal{A}} \subset \mathcal{A}V$. Thus for each $i=1, \ldots, r$
$$\overline{\mathcal{A}} \subset \mathcal{A}V_i \subset \mathcal{V}_i.$$
But $\mathcal{A}V_i = \oplus_{j=0}^{\infty} \mathcal{A}_jV_i$. Also, by Thm.\ \ref{main} $\overline{\mathcal{A}}= \cap_{i=1}^r \mathcal{V}_i \cap \mathcal{B}$. Thus, set-theoretically $\overline{\mathcal{A}_n}= \cap_{i=1}^{r} \mathcal{A}_{n}V_{i} \cap \mathcal{B}_n$. Set $K:=\mathrm{Frac}(R)$. Suppose $V_i \neq K$ for each $i=1, \ldots, r$. Showing that the decomposition is minimal and unique 
is done in the same way as in Thm.\ \ref{main}. Here we will show just the uniqueness of the valuations. Suppose there exist DVRs $V_{1}', \ldots, V_{s}'$ in $K$ such that 
\begin{equation}\label{decos2}
\overline{\mathcal{A}}= \cap_{j=1}^{s} \mathcal{A}V_{j}' \cap \mathcal{B}
\end{equation}
is minimal. As in the proof of Thm.\ \ref{main} we can assume that $\mathrm{Frac}(\mathcal{B})=\mathrm{Frac}(\overline{\mathcal{A}})$. If $V_{j}'=K$ for some $j$, then because (\ref{decos2}) is minimal we get $s=1$ and $V_{1}'=K$. So $\overline{\mathcal{A}} =\mathcal{A}K \cap \mathcal{B}$. Localizing at $\mathfrak{q}_1$ we obtain that $\mathcal{V}_1=(\mathcal{A}K)_{\mathfrak{q}_1}$. But $K \subset (\mathcal{A}K)_{\mathfrak{q}_1}$. Thus $V_1 = K$, a contradiction. Therefore, $V_{j}' \neq K$ for each $j$. Again, by localizing (\ref{decos2}) at $\mathfrak{q}_i$, we get that there is a $j$ such that $\mathcal{V}_i= (\mathcal{A}V_{j}')_{\mathfrak{q}_i}$. But $V_i=\mathcal{V}_i \cap \mathrm{Frac}(R)$ and $V_{j}' \in (\mathcal{A}V_{j}')_{\mathfrak{q}_i} \cap \mathrm{Frac}(R)$. Because $V_i \neq \mathrm{Frac}(R)$, then $V_i=V_{j}'$. Thus $r=s$ by minimality and after possible renumbering $V_i=V_{i}'$ for each $i=1, \ldots, r$.

Consider Cor.\ \ref{Rees}. Apply Cor.\ \ref{graded} with $\mathcal{A}:=R[It]$ and $\mathcal{B}:=R[t]$. In the introduction we proved that $R[It]V=\overline{R[It]}V$ for each valuation $V$ in $K$. What remains to be shown is that $V_i \neq K$ for each $i$. Indeed, the prime ideals in $\overline{R[It]}$ are contractions of extensions of prime ideals of $R$ to $R[t]$. Thus $\mathrm{ht}(\mathfrak{q}_i \cap R) \geq 1$ and so $V_i \neq K$ for each $i$ otherwise $\mathfrak{q}_{i}\overline{R[It]}_{\mathfrak{q}_i}$ is a unit ideal which is impossible.
\qedsymbol

A version of Cor.\ \ref{graded} for Rees algebras of modules is proved by Rees in \cite[Thm.\ 1.7]{Rees87}.

\vspace{.1cm}
\begin{center}
{\it Proof of Theorem \ref{Chevalley}}
\end{center}
\vspace{.2cm}
First, we show that $\mathfrak{m}_{i} \in S(\mathrm{ht}(\mathfrak{m}_{i})+e-1)$. Recall that the maximal ideal of $\mathcal{V}_i$ contracts to $\mathfrak{q}_i$ in $\overline{\mathcal{A}}$. Set $\mathfrak{p}_{i}:= \mathfrak{q}_i \cap \mathcal{A}$. Then by \cite[Thm.\ 1.1 \rm{(iii)}]{Rangachev2} $\mathrm{ht}(\mathfrak{p}_i)=1$. Consider the map $\mathrm{Proj}(\mathcal{A}_{\mathfrak{m}_i}) \rightarrow \mathrm{Spec}(R_{\mathfrak{m}_{i}})$. It's closed, surjective and of finite type. By the dimension formula (\cite[\href{http://stacks.math.columbia.edu/tag/02JX}{Tag 02JX}]{Stacks}) $\dim \mathrm{Proj}( \mathcal{A}_{\mathfrak{m}_i})=\mathrm{ht}(\mathfrak{m}_{i})+e$. But $\mathcal{A}_{\mathfrak{m}_i}$ is a local formally equidimensional ring. Because $\mathrm{ht}(\mathfrak{p}_{i}\mathcal{A}_{\mathfrak{m}_i})=1$, by \cite[Lem.\ B.4.2]{Huneke} $\dim \mathrm{Proj}(\mathcal{A} \otimes k(\mathfrak{m}_i))=\mathrm{ht}(\mathfrak{m}_{i})+e-1$. Thus $\mathfrak{m}_{i} \in S(\mathrm{ht}(\mathfrak{m}_{i})+e-1)$.

Next, suppose there exists a prime $\mathfrak{n}_i$ in $R$ with $\mathfrak{n}_i \subset \mathfrak{m}_{i}$ and $\mathfrak{n}_i \in S(\mathrm{ht}(\mathfrak{m}_{i})+e-1)$. Then $\dim \mathrm{Proj}(\mathcal{A} \otimes k(\mathfrak{n}_i)) \geq \mathrm{ht}(\mathfrak{m}_{i})+e-1$. Note that $\mathfrak{n}_i \neq (0)$ for otherwise $\mathfrak{n}_i \in S(e)$ which forces $\mathrm{ht}(\mathfrak{m}_{i})=1$, a contradiction. Because $\dim \mathrm{Proj}( \mathcal{A}_{\mathfrak{n}_i})=\mathrm{ht}(\mathfrak{n}_i)+e$ then $\dim \mathrm{Proj}( \mathcal{A} \otimes k(\mathfrak{n}_{i})) \leq \mathrm{ht}(\mathfrak{n}_i)+e-1$. Therefore, $$ \mathrm{ht}(\mathfrak{m}_{i})+e-1 \leq \mathrm{ht}(\mathfrak{n}_{i})+e-1.$$
But $\mathfrak{n}_i \subset \mathfrak{m}_{i}$. Thus $\mathfrak{n}_i = \mathfrak{m}_{i}$ and $\mathfrak{m}_{i}$ is a minimal prime in $S(\mathrm{ht}(\mathfrak{m}_{i})+e-1)$. This completes the proof of Thm.\ \ref{Chevalley}. \qedsymbol

Thm.\ \ref{Chevalley} generalizes \cite[Thm.\ 7.8]{Rangachev}, which is a result for Rees algebras of modules. To see that Thm.\ \ref{Chevalley} is sharp, let $(R,\mathfrak{m})$ be a Noetherian regular local ring of dimension at least $2$, and let $h \in \mathfrak{m}$ be an irreducible element. Let $\mathcal{B}$ be the polynomial ring $R[y_1, \ldots, y_{e+1}]$ for some $e \geq 0$. Set $\mathcal{A}:=R[hy_1, \ldots, hy_{e+1}]$. Thus $\mathcal{A}$ is a polynomial subring of $\mathcal{B}$. It is normal because $R$ is regular. In the setup of Cor.\ \ref{graded} there is only one $\mathcal{V}_1=\mathcal{A}_{\mathfrak{q}_1}$
where $\mathfrak{q}_1=h\mathcal{A}$.  Note that $S(k)=S(e)$ for all $k \geq 0$ because $\mathcal{A}$ is a polynomial ring over $R$ generated by $e+1$ elements. Thus the only minimal prime in $S(e)$ is $(0)$, whereas $\mathfrak{m}_{V_i}=(h)$ is a height one prime ideal in $R$.

\vspace{.1cm}
\begin{center}
{\it A Variant of Zariski's Main Theorem }
\end{center}
\vspace{.2cm}
Let $\mathcal{A} \subset \mathcal{B}$ be Noetherian rings. Suppose $\mathcal{B}$ is a finitely generated $\mathcal{A}$-algebra. Denote by $\overline{\mathcal{A}}$ the integral closure of $\mathcal{A}$ in $\mathcal{B}$. Denote by $\mathcal{I}_{\mathcal{B}/\overline{\mathcal{A}}}$ the intersection of all elements in $\mathrm{Ass}_{\overline{\mathcal{A}}}(\mathcal{B}/\overline{\mathcal{A}})$. The following result characterizes the support of $\mathcal{B}/\overline{\mathcal{A}}$.

\begin{proposition}\label{support} Let $\mathcal{A} \subset \mathcal{B}$ be integral domains. Suppose $\mathcal{A}$ is Noetherian and $\mathcal{B}$ is a finitely generated $\mathcal{A}$-algebra. Then $\mathbb{V}(\mathcal{I}_{\mathcal{B}/\overline{\mathcal{A}}})=\mathrm{Supp}_{\overline{\mathcal{A}}}(\mathcal{B}/\overline{\mathcal{A}}).$
\end{proposition}
\begin{proof} If $\mathrm{Frac}(\overline{\mathcal{A}}) \neq  \mathrm{Frac}(\mathcal{B})$, then $(0) \in \mathrm{Ass}_{\overline{\mathcal{A}}}(\mathcal{B}/\overline{\mathcal{A}})$ and trivially $\mathbb{V}(\mathcal{I}_{\mathcal{B}/\overline{\mathcal{A}}})=\mathrm{Supp}_{\overline{\mathcal{A}}}(\mathcal{B}/\overline{\mathcal{A}})=\mathrm{Spec}(\overline{\mathcal{A}})$. Suppose $\mathrm{Frac}(\overline{\mathcal{A}}) =  \mathrm{Frac}(\mathcal{B})$. Then by \cite[Thm.\ 1.1 \rm{(i)}]{Rangachev2} $\mathcal{I}_{\mathcal{B}/\overline{\mathcal{A}}}=\mathfrak{q}_1 \cap \ldots \cap \mathfrak{q}_s$ with $\mathrm{ht}(\mathfrak{q}_i) = 1$ for each $i$. If $\mathfrak{q} \in \mathbb{V}(\mathcal{I}_{\mathcal{B}/\overline{\mathcal{A}}})$, then $\mathfrak{q}_j \subset \mathfrak{q}$ for some $j$ and thus $\mathfrak{q}_j\mathcal{A}_{\mathfrak{q}} \in \mathrm{Ass}_{\overline{\mathcal{A}}_{\mathfrak{q}}}(\mathcal{B}_{\mathfrak{q}}/\overline{\mathcal{A}}_{\mathfrak{q}})$. Hence $\mathbb{V}(\mathcal{I}_{\mathcal{B}/\overline{\mathcal{A}}}) \subset \mathrm{Supp}_{\overline{\mathcal{A}}}(\mathcal{B}/\overline{\mathcal{A}})$. Suppose $\mathfrak{q} \subset \mathrm{Supp}_{\overline{\mathcal{A}}}(\mathcal{B}/\overline{\mathcal{A}})$. Then there is $x/y \in \mathcal{B}$ with $x,y \in \overline{\mathcal{A}}$, such that its image in $\mathcal{B}_{\mathfrak{q}}$ is not $\overline{\mathcal{A}}_{\mathfrak{q}}$. In other words, if $J:=((x) \colon_{\overline{\mathcal{A}}}\ y)$, then $J \subset \mathfrak{q}$. But by Prp.\ \ref{minimal} the minimal primes of $J$ are among the $\mathfrak{q}_i$s. Thus there exists $\mathfrak{q}_j$ such that $\mathfrak{q}_j \subset \mathfrak{q}$, i.e. 
$\mathrm{Supp}_{\overline{\mathcal{A}}}(\mathcal{B}/\overline{\mathcal{A}}) \subset \mathbb{V}(\mathcal{I}_{\mathcal{B}/\overline{\mathcal{A}}})$.  
\end{proof}
In \cite[Cor.\ 4.4.9]{Grothendieck2} Grothendieck  derives the following result as a consequence of Zariski's main theorem (ZMT): if $g \colon X \rightarrow Y$ is a birational, proper morphism of noetherian integral schemes with $Y$ normal and $g^{-1}(y)$ finite for each $y \in Y$, then $g$ is an isomorphism. Below we show 
that in the affine case we can reach the same conclusion assuming that $g$ is surjective in codimension one. To do this we do not have to appeal to ZMT. In fact, our result proves ZMT in codimension one or in the special case when $\overline{\mathcal{A}}$ is a UFD (cf.\  \cite[Prp.\ 1, pg.\ 210]{Mumford} and the discussion that follows it) as shown below.

Denote by $g \colon \mathrm{Spec}(\mathcal{B}) \rightarrow \mathrm{Spec}(\overline{\mathcal{A}})$ the induced map on ring spectra. Denote by  $\mathbb{V}(I_g)$ the Zariski closure of $\mathrm{Im}(g) \cap \mathrm{Supp}_{\overline{\mathcal{A}}}(\mathcal{B}/\overline{\mathcal{A}})$. 

\begin{theorem}\label{baby} Let $\mathcal{A} \subset \mathcal{B}$ be integral domains. Suppose $\mathcal{A}$ is Noetherian and $\mathcal{B}$ is a finitely generated $\mathcal{A}$-algebra. Assume $\mathrm{Frac}(\overline{\mathcal{A}})=\mathrm{Frac}(\mathcal{B})$.
\begin{enumerate}
    \item[\rm{(i)}] If $g$ is surjective, then $\mathcal{B}=\overline{\mathcal{A}}$.
\item[\rm{(ii)}] If $\mathcal{B} \neq \overline{\mathcal{A}}$, then $\mathrm{ht}(I_g) \geq 2$. 
\end{enumerate}
\end{theorem}
\begin{proof} Consider $\rm{(i)}$. Suppose there exists $\mathfrak{q} \in \mathrm{Ass}_{\overline{\mathcal{A}}}(\mathcal{B}/\overline{\mathcal{A}})$. Because $\mathfrak{q} \neq (0)$, by  \cite[Thm.\ 1.1 \rm{(i)}]{Rangachev2} $\overline{\mathcal{A}}_{\mathfrak{q}}$ is a DVR. But $\overline{\mathcal{A}}_{\mathfrak{q}} \neq \mathcal{B}_{\mathfrak{q}}$. Thus $\mathcal{B}_{\mathfrak{q}}=\mathrm{Frac}(\mathcal{B})$. This contradicts the assumption that there exists a prime in $\mathcal{B}$ that contracts to $\mathfrak{q}$. Thus $\mathrm{Ass}_{\overline{\mathcal{A}}}(\mathcal{B}/\overline{\mathcal{A}})$ is empty, and so by Prp.\ \ref{support} $\mathcal{B}=\overline{\mathcal{A}}$. 
Consider $\rm{(ii)}$. If $\mathfrak{q}$ is a minimal prime in $\mathrm{Supp}_{\overline{\mathcal{A}}}(\mathcal{B}/\overline{\mathcal{A}})$, then $\mathrm{ht}(\mathfrak{q})=1$ and $\mathcal{B} \otimes_{\overline{\mathcal{A}}} \kappa(\mathfrak{q})$ is empty as shown above. Thus the minimal primes of $I_g$ are of height at least $2$. 
\end{proof}

Thm.\ \ref{baby} \rm{(ii)} recovers the second part of \cite[Prp.\ 1, (2), pg.\ 210]{Mumford} without assuming that $\overline{\mathcal{A}}$ is a UFD.

\begin{definition}
Let $Q \in \mathrm{Spec}(\mathcal{B})$.  Set  $\mathfrak{p}:=Q \cap \mathcal{A}$. We say that $\mathrm{Spec}(\mathcal{B}) \rightarrow \mathrm{Spec}(\mathcal{A})$ is {\it quasi-finite} at $Q$ if $Q$ is isolated in its fiber, i.e. if the field extension $\kappa(\mathfrak{p}) \subset \kappa(Q)$ is finite and $\dim(\mathcal{B}_{Q}/\mathfrak{p}\mathcal{B}_Q)=0$. 

\end{definition}
The following two corollaries of Thm.\ \ref{baby} are special cases of ZMT. 
\begin{corollary}\label{ZMT1}
 Let $\mathcal{A} \subset \mathcal{B}$ be integral domains. Suppose $\mathcal{A}$ is Noetherian and $\mathcal{B}$ is a finitely generated $\mathcal{A}$-algebra. Let $Q \in \mathrm{Spec}(\mathcal{B})$ and set $\mathfrak{q}:=Q \cap \overline{\mathcal{A}}$. Assume $\mathrm{Spec}(\mathcal{B}) \rightarrow \mathrm{Spec}(\mathcal{A})$ is {\it quasi-finite} at $Q$ and $\mathrm{ht}(\mathfrak{q})=1$. Then there exists $f \in \mathcal{I}_{\mathcal{B}/\overline{\mathcal{A}}} $ with $f \not \in \mathfrak{q}$ such that $\mathcal{B}_f=\overline{\mathcal{A}}_f$.
\end{corollary}
\begin{proof}
Because $\mathrm{ht}(\mathfrak{q})=1$, then by Cor.\ 2 to \cite[Thm.\ 31.7]{Matsumura} $\overline{\mathcal{A}}_{\mathfrak{q}}$ is a universally catenary Noetherian ring. Applying the dimension formula for $\overline{\mathcal{A}}_{\mathfrak{q}}$ and $\mathcal{B}_{\mathfrak{q}}$ we get
\begin{equation}\label{dim f-la}
 \mathrm{ht}(Q)+\mathrm{tr.\ deg}_{\kappa(\mathfrak{q})}\kappa(Q)=\mathrm{ht}(\mathfrak{q})+\mathrm{tr.\ deg}_{\overline{\mathcal{A}}}\mathcal{B} 
\end{equation}
Because $\mathfrak{q}\mathcal{B}_{Q}$ is $Q\mathcal{B}_{Q}$-primary we have $\mathrm{ht}(\mathfrak{q})\geq \mathrm{ht}(Q)$. 
But $\kappa(Q)$ is a finite field extension of $\kappa(\mathfrak{q})$. So $\mathrm{tr.\ deg}_{\kappa(\mathfrak{q})}\kappa(Q)=0$. Thus $\mathrm{ht}(Q)=\mathrm{ht}(\mathfrak{q})$ and $\mathrm{Frac}(\overline{\mathcal{A}})=\mathrm{Frac}(\mathcal{B})$. Applying Thm.\ \ref{baby} to $\overline{\mathcal{A}}_{\mathfrak{q}}$ and $\mathcal{B}_{\mathfrak{q}}$ we get $\overline{\mathcal{A}}_{\mathfrak{q}}=\mathcal{B}_{\mathfrak{q}}$. Because $\mathrm{Frac}(\overline{\mathcal{A}})=\mathrm{Frac}(\mathcal{B})$, by Prp.\ \ref{minimal} for each $b \in \mathcal{B}$ there exists a positive integer $k_b$ such that $\mathcal{I}_{\mathcal{B}/\overline{\mathcal{A}}}^{k_b}b \in \overline{\mathcal{A}}$. Thus
for each $f \in \mathcal{I}_{\mathcal{B}/\overline{\mathcal{A}}}$ we have $\overline{\mathcal{A}}_f=\mathcal{B}_f$. By Prp.\ \ref{support} $\mathcal{I}_{\mathcal{B}/\overline{\mathcal{A}}} \not\subset \mathfrak{q}$. So we can select $f \in \mathcal{I}_{\mathcal{B}/\overline{\mathcal{A}}}$ with $f \not \in \mathfrak{q}$.
\end{proof}

\begin{corollary}\label{Mum} Let $\mathcal{A} \subset \mathcal{B}$ be integral domains. Suppose $\mathcal{A}$ is Noetherian and $\mathcal{B}$ is a finitely generated $\mathcal{A}$-algebra. Assume $\mathcal{A}$ is universally catenary and $\overline{\mathcal{A}}$ is a UFD.  Let $Q \in \mathrm{Spec}(\mathcal{B})$ and set $\mathfrak{q}:=Q \cap \overline{\mathcal{A}}$. Suppose $\mathrm{Spec}(\mathcal{B}) \rightarrow \mathrm{Spec}(\mathcal{A})$ is {\it quasi-finite} at $Q$. Then there exists $f \in \mathcal{I}_{\mathcal{B}/\overline{\mathcal{A}}} $ with $f \not \in \mathfrak{q}$ such that $\mathcal{B}_f=\overline{\mathcal{A}}_f$.
\end{corollary}
\begin{proof}
An application of the dimension formula as in (\ref{dim f-la}) for $\mathcal{A}$, $\mathcal{B}$, $Q$ and $\mathfrak{p}:=Q \cap \mathcal{A}$ implies that $\mathrm{ht}(Q)=\mathrm{ht}(\mathfrak{p})$ and $\mathrm{Frac}(\mathcal{B})=\mathrm{Frac}(\overline{\mathcal{A}})$. By \cite[Prp.\ 4.8.6]{Huneke} $\mathrm{ht}(\mathfrak{q})=\mathrm{ht}(Q)$.
We proceed by induction on $\mathrm{ht}(\mathfrak{q})$. For the height zero case simply set $f$ to be the common denominator of the generators of $\mathcal{B}$ as an $\overline{\mathcal{A}}$-algebra. The case $\mathrm{ht}(\mathfrak{q})=1$ was handled in Cor.\ \ref{ZMT1}. Suppose $\mathrm{ht}(\mathfrak{q})=n$. Let $\mathfrak{q}_1$ be a prime ideal in $\overline{\mathcal{A}}$ of height one. Because $\overline{\mathcal{A}}$ is a UFD, there exists $x_1 \in \overline{\mathcal{A}}$ such that $\mathfrak{q}_1 = (x_1)$. Because $\mathcal{A}$ is universally catenary and $\overline{\mathcal{A}}$ is integral over $\mathcal{A}$, then by \cite[Thm.\ 3.1]{Ratliff} there exists a chain of prime ideals $(0) \subset \mathfrak{q}_1 \subset \ldots  \subset \mathfrak{q}_{n}:=\mathfrak{q}$. Assume $\overline{\mathcal{A}}$ is local at $\mathfrak{q}$. Let $x_1, x_2, \ldots, x_n$ be a system of parameters for $\mathfrak{q}$. Consider $(x_1, \ldots, x_{n-1})\mathcal{B}$. Then there exists a prime $Q' \subset Q$ minimal over $(x_1, \ldots, x_{n-1})\mathcal{B}$.  Since $\mathrm{ht}(Q)=n$, then by Krull's height theorem $Q' \neq Q$. Set $\mathfrak{q}':=Q' \cap \overline{\mathcal{A}}$. Then $\mathfrak{q}' \neq \mathfrak{q}$ because $Q$ is isolated in its fiber. Thus $\mathrm{ht}(\mathfrak{q}')=n-1$ and $Q'$ is isolated in its fiber over $Q' \cap \mathcal{A}$. Also, $\mathfrak{q}_1 \subset \mathfrak{q}$.  By the induction hypothesis $\overline{\mathcal{A}}_{\mathfrak{q}'}=\mathcal{B}_{\mathfrak{q}'}$. In particular, $\mathfrak{q}_1 \not \in \mathrm{Ass}_{\overline{\mathcal{A}}}(\mathcal{B}/\overline{\mathcal{A}})$. Therefore, $\overline{\mathcal{A}}_{\mathfrak{q}}=\mathcal{B}_{\mathfrak{q}}$ by Prp.\ \ref{support}. The rest follows as in the proof of Cor.\ \ref{ZMT1}. 
\end{proof}

%Cor.\ \ref{Mum} recovers \cite[Prp.\ 1 (1), pg.\ 210]{Mumford} where it's assumed that $\mathcal{B}$ and $\overline{\mathcal{A}}$ have the same field of fractions with the weaker assumption that $\mathrm{Spec}(\mathcal{B}) \rightarrow \mathrm{Spec}(\mathcal{A})$ is quasi-finite at a prime in $\mathcal{B}$. 

\end{document}